\documentclass[11pt]{amsart}
\usepackage[margin=1in]{geometry}
\usepackage[T1]{fontenc}
\usepackage{lmodern}
\usepackage{amsmath,amssymb,amsthm,mathtools}
\usepackage{xcolor}
\usepackage{enumitem}
\usepackage{microtype}
\usepackage{hyperref}
\hypersetup{
  colorlinks=true,
  linkcolor=blue,
  citecolor=blue,
  urlcolor=blue,
  pdftitle={Vertex volumes, lattice-minima tails, and height zeta functions for PGLd},
  pdfauthor={Soonki Hong and Sanghoon Kwon},
  pdfkeywords={Bruhat-Tits buildings, arithmetic quotients, vertex volumes, lattice minima, height zeta functions}
}

\newcommand{\Fq}{\mathbb F_q}
\newcommand{\OO}{\mathcal O}
\newcommand{\ZZ}{\mathbb Z}
\newcommand{\RR}{\mathbb R}
\newcommand{\CC}{\mathbb C}
\newcommand{\GL}{\operatorname{GL}}
\newcommand{\PGL}{\operatorname{PGL}}
\newcommand{\SL}{\operatorname{SL}}

\newcommand{\diag}{\operatorname{diag}}
\newcommand{\vol}{\operatorname{vol}}
\newcommand{\covol}{\operatorname{covol}}
\newcommand{\supp}{\operatorname{supp}}
\newcommand{\Comp}{\operatorname{Comp}}
\newcommand{\Cuts}{\operatorname{Cuts}}
\newcommand{\Blocks}{\operatorname{Blocks}}

\theoremstyle{plain}
\newtheorem{theorem}{Theorem}[section]
\newtheorem{proposition}[theorem]{Proposition}
\newtheorem{lemma}[theorem]{Lemma}
\newtheorem{corollary}[theorem]{Corollary}

\theoremstyle{definition}

\newtheorem{remark}[theorem]{Remark}
\newtheorem{example}[theorem]{Example}

\usepackage[nameinlink,capitalize]{cleveref}

\numberwithin{equation}{section}
\allowdisplaybreaks

\title[Vertex volumes and height zeta functions for $\PGL_d$]{Vertex volumes, lattice-minima tails, and height zeta functions for the standard arithmetic quotient of $\PGL_d$}
\author{Soonki Hong and Sanghoon Kwon}
\date{July 30, 2026}
\thanks{This work was supported by the Basic Science Research Program through the National Research Foundation of Korea (NRF), funded by the Ministry of Education (grant no. RS-2025-25415913).}

\begin{document}

\begin{abstract}
We study the standard nonuniform arithmetic quotient of the affine
Bruhat--Tits building attached to
\(\PGL_d(\mathbb F_q(\!(t^{-1})\!))\), with Haar measure normalized so
that a maximal compact subgroup has volume one.  We first compute its
vertex volume in closed product form.  The proof is entirely
building-theoretic: vertices are parametrized by a dominant sector,
their stabilizers are counted exactly, and the resulting sum over
block compositions is evaluated by a cut-set recursion.

On the same quotient, we introduce a homothety-invariant normalized
lattice-minima height \(\alpha\).  We determine its exact integrability
threshold, proving that \(\alpha\) belongs to \(L^r\) precisely for
\(0<r<d\), and establish a sharp cusp-tail estimate of order
\(T^{-d}\).  The associated positive-moment height zeta function,
equivalently the Mellin transform of the cusp-height distribution,
converges exactly in the half-plane
\(\operatorname{Re}(s)<d\).  It admits a meromorphic continuation as a
rational function of \(q^{s/d}\) and has a simple pole at \(s=d\), with
an explicit critical coefficient.  We also compute the resulting
rational functions explicitly for \(d=3,4,5\).  Thus the same
dominant-sector coordinates simultaneously control volume, cusp
decay, and the analytic structure of the height zeta function.
\end{abstract}


\maketitle
\tableofcontents

\section{Introduction}

Let $q$ be a prime power and fix an integer $d\ge2$.  We use the non-archimedean local
field
\(
  F=\Fq(\!(t^{-1})\!),
\)
with valuation ring
\(
  \OO=\Fq[\![t^{-1}]\!],
\)
and absolute value normalized by $|t|=q$.  The main locally compact group, maximal compact
subgroup, and arithmetic subgroup are
\[
  G=\PGL_d(F),\qquad K=\PGL_d(\OO),\qquad \Gamma=\PGL_d(\Fq[t]).
\]
The affine Bruhat--Tits building of $G$ is denoted by $\mathcal B_d$.  Its vertices are
identified with $G/K$.  We normalize Haar measure on $G$ by
\[
  \vol(K)=1.
\]
We refer to Bruhat--Tits~\cite{BruhatTits} and Abramenko--Brown
\cite{AbramenkoBrown} for the structure and standard terminology of affine buildings.
The induced measure on the vertex quotient
\[
  Y=\Gamma\backslash G/K=\Gamma\backslash\mathcal B_d
\]
is denoted by $\nu$.  Thus $\nu(Y)$ is the vertex-volume of the quotient.  We write
$\mu=\nu/\nu(Y)$ for the associated probability measure.  Throughout the paper, the
notations $X\ll Y$ and $Y\gg X$ mean that $X\le CY$ for a constant $C>0$ depending
only on $d$ and $q$.  We write $X\asymp Y$ when both $X\ll Y$ and $Y\ll X$ hold.

The first main result is the following explicit formula.

\begin{theorem}[Vertex volume]\label{thm:volume}
For every $d\ge2$,
\[
  \nu(Y)=\vol(\Gamma\backslash \mathcal B_d)
  =
  \frac{d}{\left(\prod_{m=1}^{d}(q^m-1)\right)
            \left(\prod_{m=2}^{d-1}(q^m-1)\right)}.
\]
The second product in the denominator is interpreted as $1$ when $d=2$.
\end{theorem}

The next result concerns a normalized lattice-minima function.  Let
\[
  \Lambda_0=\Fq[t]^d\subset F^d
\]
be the standard $\Fq[t]$-lattice.  We use the max norm on $F^d$ induced by $|t|=q$, and
on $\bigwedge^jF^d$ we use the associated sup norm in the standard exterior basis.  If
$\Lambda$ is a rank-$d$ $\Fq[t]$-lattice in $F^d$, let $\covol(\Lambda)$ be normalized by
$\covol(\Lambda_0)=1$ and by
\[
  \covol(g\Lambda_0)=|\det g|
\]
for $g\in\GL_d(F)$.  This convention is used only to make the following function invariant
under homotheties.  For the homothety class $[\Lambda]$ define
\begin{equation}\label{eq:alpha-def-intro}
  \alpha([\Lambda])
  =\max_{1\le j\le d}
    \left\{
    \covol(\Lambda)^{j/d}
    \|v_1\wedge\cdots\wedge v_j\|^{-1}:
    v_1,\ldots,v_j\in\Lambda \text{ are $F$-linearly independent}
    \right\}.
\end{equation}
For unimodular lattices, i.e. $\covol(\Lambda)=1$, this is the usual lattice-minima
function.  We define the function on
\[
  Y=\Gamma\backslash G/K
\]
through the inversion map
\[
  \Gamma gK\longmapsto Kg^{-1}\Gamma.
\]
Namely, $\alpha(\Gamma gK)$ is the usual $K$-invariant height of the homothety class
$[g^{-1}\Lambda_0]$.  Since $K$ acts by isometries and $\Gamma$ preserves
$\Lambda_0$, this is well-defined.

\begin{theorem}[Critical integrability exponent]\label{thm:integrability}
For every real $r>0$,
\[
  \int_Y \alpha(x)^r\,d\nu(x)<\infty
  \quad\Longleftrightarrow\quad
  0<r<d.
\]
The same criterion holds with the probability measure $\mu$ in place of $\nu$.
\end{theorem}

The proof of both theorems uses the same exponent coordinates.  For the volume calculation
we use the $\PGL_d$ homothety normalization
\[
  m=(m_1,\ldots,m_d),\qquad m_1\ge\cdots\ge m_d=0.
\]
For the height calculation we pass from $m$ to the trace-zero vector
\[
  n_i=m_i-\frac1d\sum_{\ell=1}^d m_\ell.
\]
Keeping these two normalizations separate is essential: the first gives actual double-coset
representatives in $\PGL_d$, while the second gives the determinant-normalized exponents
that govern the homothety-invariant height.

We next attach a zeta function to the cusp height.  Using the
positive-moment convention, define
\begin{equation}\label{eq:zeta-s}
  Z_\alpha(s)=\int_Y \alpha(x)^s\,d\nu(x),
  \qquad s\in\CC.
\end{equation}
Equivalently, \(Z_\alpha\) is the Mellin transform of the
cusp-height distribution determined by \(\alpha\) and \(\nu\). See Section~\ref{sec:hzf}.

\begin{theorem}[Height zeta function]\label{thm:height-zeta-main}
The height zeta function \(Z_\alpha(s)\) converges absolutely if and only if
\[
  \operatorname{Re}(s)<d.
\]
It extends meromorphically to all \(s\in\CC\) as a rational function of
\[
  u=q^{s/d}.
\]
At \(s=d\), it has a simple pole, and
\[
  Z_\alpha(s)=\frac{\kappa_{d,q}}{d-s}+O(1)
  \qquad (s\to d^-),
\]
where \(\kappa_{d,q}>0\) is given explicitly by
Corollary~\ref{cor:critical-pole-coefficient}.
\end{theorem}

\subsection*{Context and novelty}

The covolume formula in Theorem~\ref{thm:volume} is compatible with the general theory of
arithmetic quotients. Prasad's volume formula treats absolutely quasi-simple simply
connected groups over global fields; related semisimple quotients, such as $\PGL_d$, are
obtained after the usual central-isogeny bookkeeping and after the global and local Haar
normalizations have been specified~\cite{Prasad}.  Harder's Gauss--Bonnet formula gives
another conceptual route through the Euler--Poincar\'e measure~\cite{Harder}.  We do not
claim a new general covolume formula.

Minimum-covolume problems over non-Archimedean local fields were studied in rank one by
Lubotzky~\cite{LubotzkyMinimal} and, for Chevalley groups in positive characteristic, by
Salehi Golsefidy~\cite{SalehiGolsefidy,SalehiGolsefidyNonuniform}.  Those works vary the
lattice and seek global minimizers, whereas Theorem~\ref{thm:volume} fixes the standard
arithmetic lattice and computes its \(K\)-vertex volume in a specified local Haar
normalization.

Dominant-sector coordinates and low-dimensional stabilizer weights for the same
standard quotient appeared in our earlier analyses of the cases \(d=3\) and \(d=4\)
\cite{HongKwonPGL3,HongKwonPGL4}.  More generally, Sela--Schaps--Vishne determined
fundamental domains for congruence subgroups of the standard non-uniform lattice and
expressed finite covolume through stabilizer weights \cite{SelaSchapsVishne}.  Thus the
sector and stabilizer sum themselves are not asserted to be new here.  On the volume side,
our new contribution is the closed evaluation, in the normalization \(\vol(K)=1\), of the
resulting all-dimensional composition sum by a canonical cut-set recursion.  The same
simple-root coordinates then yield the sharp \(L^r\) threshold, the two-sided \(T^{-d}\)
tail, and the rational height zeta function.

\begin{enumerate}[label=(\roman*)]
\item We compute the $K$-vertex volume of
$\PGL_d(\Fq[t])\backslash\PGL_d(\Fq(\!(t^{-1})\!))/\PGL_d(\Fq[\![t^{-1}]\!])$ directly in the
normalization $\vol(K)=1$, without passing through Tamagawa or Euler--Poincar\'e
normalizations.
\item After the standard sector decomposition is fixed, the remaining volume
computation is an explicit building-side summation: exact stabilizer counts, a fixed-block
geometric summation, and a canonical cut-set recursion on the path \(0\to d\).
\item The same simple-root difference coordinates simultaneously control the normalized
lattice-minima function.  They give not only the sharp criterion
$\alpha\in L^r\Longleftrightarrow r<d$, but also a sharp $T^{-d}$ cusp-tail estimate, an
explicit critical pole coefficient, and rational height zeta functions.
\item In the cases \(d=3,4,5\), the rational-cone formula gives explicit closed
formulas for \(Z_{\alpha,3}\), \(Z_{\alpha,4}\), and \(Z_{\alpha,5}\).  These provide test
cases for a more systematic theory of height zeta functions on affine-building quotients.
\end{enumerate}

When \(d=2\), the building is a tree and the stabilizer-inverse volume sum is the familiar
graph-of-groups covolume formula; see Serre~\cite{Serre} and Bass--Lubotzky
\cite{BassLubotzky}.  Lubotzky's structure theory for rank-one lattices over local fields
provides further context for non-uniform cusps in positive characteristic
\cite{LubotzkyRankOne}.  The present argument may be viewed as a higher-rank,
normalization-sensitive analogue for this particular arithmetic quotient.

Arithmetic quotients of Bruhat--Tits buildings for projective general linear groups in
positive characteristic have also been studied from automorphic and modular-symbol
viewpoints~\cite{KondoYasuda}.  Those works are complementary to the present calculation: they explain the
representation-theoretic and homological geometry of the quotient, while our focus is the
explicit vertex-volume/height summation in the standard sector.

For the broader function-field reduction-theoretic and automorphic background, see
Harder~\cite{HarderReduction,HarderChevalley}; for the geometric and homological use of
buildings in the function-field arithmetic setting, see also Borel--Serre and
Stuhler~\cite{BorelSerre,Stuhler}.

The integrability theorem is likewise consistent with the general homogeneous-dynamics
picture.  Logarithm laws and cusp excursions on spaces of lattices are classical in the
real case \cite{KleinbockMargulis}; the treatment of unbounded test functions in
\cite{KleinbockShiWeiss} provides a complementary integrability perspective.  Ultrametric
logarithm laws and related local-field equidistribution results are given in
\cite{AthreyaGhoshPrasadI,AthreyaGhoshPrasadII,KwonLim}.  The point here is that the
standard quotient for $\PGL_d(\Fq(\!(t^{-1})\!))$ admits an exact
sector-coordinate summation rather than only a qualitative non-divergence estimate.

Quantitative non-divergence on products of real and non-Archimedean homogeneous spaces is
developed systematically by Kleinbock--Tomanov~\cite{KleinbockTomanov}.  Our theorem is
more specialized, but in the present discrete sector it identifies the exact critical
exponent and the precise power of the cusp tail.

\section{Vertex volume and the cut-set summation}

We first describe the dominant sector and its stabilizer weights, and then evaluate
the resulting composition sum by the cut-set recursion.  Together, these steps prove
Theorem~\ref{thm:volume}.

\subsection{The standard sector}

For an integer vector $m=(m_1,\ldots,m_d)\in\ZZ^d$, define
\[
  x_m=\diag(t^{m_1},\ldots,t^{m_d})\in\PGL_d(F).
\]
The following standard reduction fixes the sector used throughout the paper.
It is the explicit \(\PGL_d\) instance of the reduction-theoretic sector appearing in
Harder's function-field reduction theory~\cite{HarderReduction}; the proof below uses the
equivalent splitting description of vector bundles on \(\mathbb P^1\).

\begin{lemma}[Dominant sector]\label{lem:dominant-sector}
The map
\begin{equation}\label{eq:dominant-m}
  m=(m_1,\ldots,m_d),\qquad m_1\ge\cdots\ge m_d=0,
  \qquad\longmapsto\qquad \Gamma x_mK
\end{equation}
is a bijection from the displayed set of integer vectors onto the double quotient
$\Gamma\backslash G/K$.
\end{lemma}

\begin{proof}
We first work with $\GL_d$ and then pass to $\PGL_d$.  Let $A=\Fq[t]$ and
$\pi=t^{-1}$, so that $F=\Fq(\!(\pi)\!)$ and $\OO=\Fq[\![\pi]\!]$.  A matrix
$g\in\GL_d(F)$ defines a vector bundle $\mathcal E_g$ on $\mathbb P^1_{\Fq}$ by gluing the
trivial bundle on $\operatorname{Spec}A$ to the trivial bundle on the formal disc at
$\infty$ through the transition function $g$ on the punctured disc. This formal
gluing is the vector-bundle case of the Beauville--Laszlo descent lemma
\cite{BeauvilleLaszlo}. Multiplication of $g$
on the left by $\GL_d(A)$ and on the right by $\GL_d(\OO)$ only changes the two
trivializations, so the double coset records the isomorphism class of this vector bundle
with its two local trivializations forgotten.

By the Birkhoff--Grothendieck theorem over $\mathbb P^1$~\cite{Grothendieck}
(in its algebraic form over an arbitrary field; see Hazewinkel-Martin \cite{HazewinkelMartin}), there are uniquely determined
integers $m_1\ge\cdots\ge m_d$ such that
\[
  \mathcal E_g\simeq \mathcal O(m_1)\oplus\cdots\oplus\mathcal O(m_d).
\]
Equivalently, after changing the two trivializations, $g$ is represented by
$\diag(t^{m_1},\ldots,t^{m_d})$.  This gives representatives for the corresponding
$\GL_d(A)\backslash\GL_d(F)/\GL_d(\OO)$ double quotient, unique up to the displayed order.
Passing to $\PGL_d$ identifies tuples differing by a common additive constant, because
multiplication by the scalar matrix $t^c I_d$ adds $c$ to all exponents.  We therefore choose
the unique representative of this scalar class with $m_d=0$.  This proves both existence and
uniqueness in $\Gamma\backslash G/K$.
\end{proof}

The vector in Lemma~\ref{lem:dominant-sector} is called the dominant representative of the
double coset.  The stabilizer of the corresponding vertex is
\begin{equation}\label{eq:Gamma-m}
  \Gamma_m=\Gamma\cap x_mKx_m^{-1}.
\end{equation}

\begin{lemma}[Stabilizer-inverse volume formula]\label{lem:orbit-stabilizer-volume}
With the Haar normalization $\vol(K)=1$,
\begin{equation}\label{eq:volume-stabilizer-sum}
  \nu(Y)=\sum_m\frac1{|\Gamma_m|},
\end{equation}
where $m$ runs over all vectors satisfying $m_1\ge\cdots\ge m_d=0$.
\end{lemma}

\begin{proof}
Let $x=x_m$ and consider the compact open subset $xK\subset G$.  Its image in
$\Gamma\backslash G$ has measure
\[
  \frac{\vol(xK)}{|\Gamma\cap xKx^{-1}|}=\frac{1}{|\Gamma_m|}.
\]
Indeed, the map $K\to\Gamma\backslash G$, $k\mapsto\Gamma xk$, has fibers exactly the
left cosets of the finite group $x^{-1}\Gamma x\cap K$, whose order is $|\Gamma_m|$; the
finiteness follows because $\Gamma_m$ is a discrete subgroup of the compact group
$xKx^{-1}$.  Pushing forward to the vertex quotient $\Gamma\backslash G/K$ assigns this
mass to the vertex represented by $\Gamma xK$.  Summing over the dominant representatives
from Lemma~\ref{lem:dominant-sector} gives \eqref{eq:volume-stabilizer-sum}.
\end{proof}

\subsection{Block data}

A composition of $d$ is a tuple $a=(a_1,\ldots,a_k)$ of positive integers with
$a_1+\cdots+a_k=d$; the set of all such compositions is denoted by $\Comp(d)$.  Given a
composition, set
\[
  s_i=a_1+\cdots+a_i\qquad (1\le i\le k),\qquad s_0=0.
\]
The cut set of $a$ is
\[
  \Cuts(a)=\{s_1,\ldots,s_{k-1}\}\subset\{1,\ldots,d-1\}.
\]

For a dominant vector $m$, group equal entries into blocks.  Thus there are a composition
$a=(a_1,\ldots,a_k)$ of $d$ and strictly decreasing heights
\[
  b_1>b_2>\cdots>b_{k-1}>b_k=0
\]
such that
\[
  m_{s_{i-1}+1}=\cdots=m_{s_i}=b_i\qquad (1\le i\le k).
\]
The positive height differences are
\begin{equation}\label{eq:y-differences}
  y_i=b_i-b_{i+1}\ge1\qquad (1\le i\le k-1).
\end{equation}
Then $b_i=y_i+\cdots+y_{k-1}$.

\subsection{Stabilizer weights and fixed-block summation}

Let
\begin{equation}\label{eq:Qn}
  Q_0=1,
  \qquad
  Q_n=\prod_{j=1}^n(q^j-1)\quad (n\ge1).
\end{equation}
Thus
\begin{equation}\label{eq:GL-size}
  |\GL_n(\Fq)|=q^{\binom n2}Q_n.
\end{equation}

\begin{lemma}[Stabilizer count]\label{lem:stabilizer}
Let $m$ have block data $(a,b)$ as above.  Then
\[
  |\Gamma_m|
  =
  \frac{\prod_{i=1}^k |\GL_{a_i}(\Fq)|}{q-1}
  \cdot
  q^{\sum_{1\le i<j\le k}a_ia_j(b_i-b_j+1)}.
\]
Equivalently,
\begin{equation}\label{eq:inverse-stabilizer-weight}
  \frac1{|\Gamma_m|}
  =
  q^{-\binom d2}(q-1)
  \left(\prod_{i=1}^k Q_{a_i}^{-1}\right)
  q^{-\sum_{r=1}^{k-1}s_r(d-s_r)y_r}.
\end{equation}
\end{lemma}

\begin{proof}
The condition $\gamma\in\Gamma_m$ is $x_m^{-1}\gamma x_m\in K$.  Although this is a
condition in $\PGL_d$, no additional scalar changes the degree bounds.  Indeed, choose a
representative $\gamma\in\GL_d(\Fq[t])$; then $\det\gamma\in\Fq^\times$.  If
$\lambda x_m^{-1}\gamma x_m$ represents an element of $\GL_d(\OO)$ for some
$\lambda\in F^\times$, then its determinant has absolute value $1$, hence
$|\lambda|^d=1$ and $\lambda\in\OO^\times$.  Thus the projective condition is
equivalent to the ordinary integrality condition $x_m^{-1}\gamma x_m\in\GL_d(\OO)$
up to an $\OO^\times$-scalar.

In block form, entries strictly below the diagonal vanish, diagonal blocks are constant
invertible matrices, and the $(i,j)$-block with $i<j$ is an arbitrary $a_i\times a_j$
matrix with entries in polynomials of degree at most $b_i-b_j$.  Therefore the diagonal
Levi factor is
\[
  \left(\prod_{i=1}^k\GL_{a_i}(\Fq)\right)/\Fq^\times,
\]
so we divide by $q-1$ once, not once per block.  The strictly upper blocks contribute
\[
  q^{\sum_{i<j}a_ia_j(b_i-b_j+1)}.
\]
Using \eqref{eq:GL-size}, the constant part of the inverse stabilizer is
\[
  (q-1)q^{-\sum_i\binom{a_i}{2}-\sum_{i<j}a_ia_j}
  \prod_iQ_{a_i}^{-1}
  =q^{-\binom d2}(q-1)\prod_iQ_{a_i}^{-1}.
\]
Finally, a difference $y_r=b_r-b_{r+1}$ separates the first $s_r$ coordinates from the last
$d-s_r$ coordinates, so
\[
  \sum_{i<j}a_ia_j(b_i-b_j)
  =\sum_{r=1}^{k-1}s_r(d-s_r)y_r.
\]
This proves \eqref{eq:inverse-stabilizer-weight}.
\end{proof}

\begin{corollary}[Fixed block type]\label{cor:fixed-block}
For a fixed composition $a=(a_1,\ldots,a_k)\in\Comp(d)$,
\[
  V(a):=\sum_{m:\,a(m)=a}\frac1{|\Gamma_m|}
\]
is equal to
\begin{equation}\label{eq:Va}
  V(a)
  =q^{-\binom d2}(q-1)
  \left(\prod_{i=1}^k Q_{a_i}^{-1}\right)
  \left(\prod_{i=1}^{k-1}(q^{s_i(d-s_i)}-1)^{-1}\right).
\end{equation}
\end{corollary}

\begin{proof}
By \eqref{eq:inverse-stabilizer-weight} and $y_i\ge1$,
\[
  \sum_{y_i\ge1}q^{-s_i(d-s_i)y_i}
  =\frac{q^{-s_i(d-s_i)}}{1-q^{-s_i(d-s_i)}}
  =(q^{s_i(d-s_i)}-1)^{-1}.
\]
Multiplying these geometric series over all cuts gives \eqref{eq:Va}.
\end{proof}

\subsection{The cut-set dynamic program}

We now evaluate the remaining finite sum over compositions and thereby complete the
volume calculation.  Define
\begin{equation}\label{eq:CdAB}
  C_d=q^{-\binom d2}(q-1),
  \qquad
  A(\ell)=Q_\ell^{-1}\quad (1\le\ell\le d),
  \qquad
  B(s)=(q^{s(d-s)}-1)^{-1}\quad (1\le s\le d-1).
\end{equation}
For $a=(a_1,\ldots,a_k)\in\Comp(d)$ define its canonical weight by
\begin{equation}\label{eq:canonical-weight}
  W(a)=\left(\prod_{i=1}^k A(a_i)\right)
       \left(\prod_{s\in\Cuts(a)}B(s)\right).
\end{equation}
Then Corollary~\ref{cor:fixed-block} gives
\begin{equation}\label{eq:volume-Cd-Wsum}
  \nu(Y)=C_d\sum_{a\in\Comp(d)}W(a).
\end{equation}

\begin{proposition}[Cut-set dynamic program]\label{prop:cut-dp}
Define $F_d(0),F_d(1),\ldots,F_d(d-1)$ by
\begin{equation}\label{eq:Fd-recursion}
  F_d(0)=1,
  \qquad
  F_d(n)=B(n)\sum_{\ell=1}^nF_d(n-\ell)A(\ell)
  \quad (1\le n\le d-1).
\end{equation}
Then
\begin{equation}\label{eq:terminal-DP}
  \sum_{a\in\Comp(d)}W(a)
  =\sum_{\ell=1}^dF_d(d-\ell)A(\ell).
\end{equation}
\end{proposition}

\begin{proof}
A composition $a=(a_1,\ldots,a_k)$ is equivalently a path
\[
  0=s_0<s_1<\cdots<s_{k-1}<s_k=d,
  \qquad a_i=s_i-s_{i-1}.
\]
For $n<d$, let $F_d(n)$ be the total weight of paths from $0$ to $n$, including the cut
weight $B(n)$ at the terminal point $n$.  Decomposing such a path by its last block length
$\ell$ gives exactly \eqref{eq:Fd-recursion}.  For a full path ending at $d$, the last block has
length $\ell$, the initial path ends at $d-\ell$ and contributes $F_d(d-\ell)$, and the final
block contributes $A(\ell)$.  Since $d$ is not a cut, there is no factor $B(d)$.  Summing over
$\ell$ proves \eqref{eq:terminal-DP}.
\end{proof}

\begin{lemma}[Closed form for the dynamic program]\label{lem:closed-F}
For $0\le n\le d-1$,
\begin{equation}\label{eq:Fd-closed}
  F_d(n)=\frac{q^{\binom n2}Q_{d-n-1}}{Q_nQ_{d-1}}.
\end{equation}
Consequently,
\begin{equation}\label{eq:Wsum-closed}
  \sum_{a\in\Comp(d)}W(a)
  =\frac{dq^{\binom d2}}{Q_dQ_{d-1}}.
\end{equation}
\end{lemma}

\begin{proof}
The formula \eqref{eq:Fd-closed} is immediate for $n=0$.  Substituting the right-hand side
of \eqref{eq:Fd-closed} into the recursion \eqref{eq:Fd-recursion} reduces the induction
step to the finite convolution
\begin{equation}\label{eq:q-conv-induction}
  \sum_{r=0}^{n-1}
  \frac{q^{\binom r2}Q_{d-r-1}}{Q_rQ_{n-r}}
  =
  \frac{q^{\binom n2}Q_{d-n-1}}{Q_n}
  (q^{n(d-n)}-1).
\end{equation}
which is Lemma~\ref{lem:appendix-first-q-identity}.  This proves \eqref{eq:Fd-closed} by
induction.

Using Proposition~\ref{prop:cut-dp} and \eqref{eq:Fd-closed},
\[
  \sum_{a\in\Comp(d)}W(a)
  =
  \frac1{Q_{d-1}}
  \sum_{n=0}^{d-1}\frac{q^{\binom n2}}{Q_n(q^{d-n}-1)}.
\]
The remaining finite identity
\begin{equation}\label{eq:q-conv-terminal}
  \sum_{n=0}^{d-1}\frac{q^{\binom n2}}{Q_n(q^{d-n}-1)}
  =\frac{dq^{\binom d2}}{Q_d}.
\end{equation}
is Lemma~\ref{lem:appendix-terminal-q-identity}.  This gives \eqref{eq:Wsum-closed}.
\end{proof}

\begin{proof}[Proof of Theorem~\ref{thm:volume}]
By \eqref{eq:volume-Cd-Wsum} and \eqref{eq:Wsum-closed},
\[
  \nu(Y)
  =q^{-\binom d2}(q-1)\frac{dq^{\binom d2}}{Q_dQ_{d-1}}
  =\frac{d(q-1)}{Q_dQ_{d-1}}.
\]
Since $Q_{d-1}=(q-1)\prod_{m=2}^{d-1}(q^m-1)$, this is exactly the stated formula.
\end{proof}

\section{The alpha function in dominant coordinates}

\subsection{Trace-zero normalization and the diagonal formula}

Recall that \(\alpha\) was defined in \eqref{eq:alpha-def-intro} on homothety
classes of \(\Fq[t]\)-lattices and transferred to \(Y\) through the inversion map
\(\Gamma gK\mapsto Kg^{-1}\Gamma\).  We now express this height in dominant-sector
coordinates.

The volume calculation used the $\PGL_d$ representative $m_d=0$.  The normalized height
uses the trace-zero vector
\begin{equation}\label{eq:n-trace-zero}
  n_i=m_i-\bar m,
  \qquad
  \bar m=\frac1d\sum_{\ell=1}^dm_\ell.
\end{equation}
Then
\[
  n_1\ge\cdots\ge n_d,
  \qquad
  \sum_{i=1}^dn_i=0.
\]
The $n_i$ may lie in $d^{-1}\ZZ$, but all differences $n_i-n_{i+1}=m_i-m_{i+1}$ are
integers.

For $0\le j\le d$, define the partial sums
\begin{equation}\label{eq:Sj-H}
  S_j(n)=n_1+\cdots+n_j,
  \qquad S_0(n)=S_d(n)=0,
\end{equation}
and define
\begin{equation}\label{eq:H-def}
  H(n)=\max_{0\le j\le d}S_j(n).
\end{equation}
The auxiliary index \(j=0\) contributes \(S_0=0\), while the \(j=d\) term in
\eqref{eq:alpha-def-intro} is identically \(1=q^{S_d}\).  Thus adjoining \(j=0\) does not
change the maximum and records the two neutral endpoints of the partial-sum profile.

The polygonal profile \(j\mapsto S_j(n)\) is closely related to the instability polygons
used in reduction theory via semistability; compare Grayson~\cite{Grayson}.  We only use
its elementary concavity and do not invoke the general semistability formalism.

\begin{proposition}[Diagonal formula for $\alpha$]\label{prop:alpha-qH}
Let $m$ satisfy \eqref{eq:dominant-m}, and let $n$ be defined by \eqref{eq:n-trace-zero}.
For the vertex represented by $\Gamma x_mK$,
\begin{equation}\label{eq:alpha-qH}
  \alpha(\Gamma x_mK)=q^{H(n)}.
\end{equation}
\end{proposition}

\begin{proof}
By \eqref{eq:alpha-def-intro} and the inversion convention following it, the value
at \(\Gamma x_mK\) is the normalized height of the homothety class of
\(x_m^{-1}\Lambda_0\).  Let
\[
  \Lambda=x_m^{-1}\Lambda_0.
\]
A basis of \(\Lambda\) is
\[
  v_i=t^{-m_i}e_i\qquad (1\le i\le d).
\]
For fixed \(j\), the smallest possible norm of a wedge of \(j\) independent vectors in
\(\Lambda\) is
\[
  q^{-(m_1+\cdots+m_j)}.
\]
Indeed, equality is attained by \(v_1\wedge\cdots\wedge v_j\).  Conversely, if
\(w_1,\ldots,w_j\in\Lambda\) are independent, their Pl\"ucker coordinates in the basis
\(v_i\) are minors with entries in \(\Fq[t]\).  At least one such minor is nonzero and
therefore has norm at least \(1\).  Since \(m_1\ge\cdots\ge m_d\), the sum of the
exponents attached to any \(j\)-subset is at most \(m_1+\cdots+m_j\).  Hence every
nonzero wedge has norm at least \(q^{-(m_1+\cdots+m_j)}\).

The covolume of \(\Lambda=x_m^{-1}\Lambda_0\) is \(q^{-\sum_i m_i}\).  Therefore the
\(j\)-th normalized contribution to \(\alpha\) is
\[
  \covol(\Lambda)^{j/d} q^{m_1+\cdots+m_j}
  =q^{-j\bar m+m_1+\cdots+m_j}
  =q^{S_j(n)}.
\]
Taking the maximum over \(0\le j\le d\) gives \eqref{eq:alpha-qH}.
\end{proof}

\subsection{Difference coordinates and the modular exponent}

To compare the cusp height with stabilizer decay, we now pass to the simple-root
difference coordinates of type \(A_{d-1}\).  Define
\begin{equation}\label{eq:delta-coordinates}
  \delta_i=m_i-m_{i+1}=n_i-n_{i+1}\in\ZZ_{\ge0},
  \qquad 1\le i\le d-1.
\end{equation}
The support of $\delta$ is
\[
  \supp(\delta)=\{i:\delta_i>0\}\subset\{1,\ldots,d-1\}.
\]
This support is exactly the cut set of the associated composition.  If
$I=\{s_1<\cdots<s_{k-1}\}$, then the block lengths are
\[
  \Blocks(I)=(s_1,s_2-s_1,\ldots,d-s_{k-1}).
\]
For $1\le i\le d-1$, put
\begin{equation}\label{eq:c_i}
  c_i=i(d-i).
\end{equation}
The inverse stabilizer weight can then be written as
\begin{equation}\label{eq:weight-delta}
  \frac1{|\Gamma_m|}=C(\supp\delta)q^{-\Phi(\delta)},
\end{equation}
where
\begin{equation}\label{eq:C-I-Phi}
  C(I)=q^{-\binom d2}(q-1)\prod_{\ell\in\Blocks(I)}Q_\ell^{-1},
  \qquad
  \Phi(\delta)=\sum_{i=1}^{d-1}c_i\delta_i.
\end{equation}
Here $I\subset\{1,\ldots,d-1\}$ and $\Blocks(I)$ denotes the list of block lengths cut out
by $I$; for $I=\varnothing$, one has $\Blocks(I)=(d)$.

The partial sums $S_j$ are linear functions of $\delta$.  More precisely,
\begin{equation}\label{eq:inverse-Cartan-coeff}
  dS_j(\delta)=\sum_{i=1}^{d-1}b_{ji}\delta_i,
  \qquad
  b_{ji}=\min(i,j)(d-\max(i,j)).
\end{equation}
The matrix $(b_{ji}/d)$ is the inverse Cartan matrix of type $A_{d-1}$.  From this
point on, $S_j(\delta)$ denotes the linear form defined by
\eqref{eq:inverse-Cartan-coeff}; we also put
\begin{equation}\label{eq:H-delta-def}
  H(\delta)=\max_{0\le j\le d}S_j(\delta),
  \qquad S_0(\delta)=S_d(\delta)=0.
\end{equation}
This agrees with $S_j(n)$ and $H(n)$ for the trace-zero vector associated with
$\delta$.

\begin{lemma}[Modular inequality]\label{lem:Phi-dH}
For every $\delta\in\ZZ_{\ge0}^{d-1}$,
\begin{equation}\label{eq:Phi-ge-dH}
  \Phi(\delta)\ge dH(\delta).
\end{equation}
Equality holds for nonzero $\delta$ if and only if $\delta$ lies on one of the rank-one rays
$\RR_{>0}e_j$, $1\le j\le d-1$, where $e_j$ is the $j$-th standard basis vector of
$\RR^{d-1}$.
\end{lemma}

\begin{proof}
Let $n$ be the trace-zero vector associated with $\delta$.  A direct computation gives
\[
  \Phi(\delta)=\sum_{1\le i<j\le d}(n_i-n_j)=2\sum_{j=1}^{d-1}S_j(n).
\]
The sequence $S_0,S_1,\ldots,S_d$ is concave because
$S_j-S_{j-1}=n_j$ and $n_1\ge\cdots\ge n_d$.  Also $S_0=S_d=0$.  If $S_r=H$, concavity
forces the graph of $S_j$ to lie above the discrete tent function with vertices
$(0,0),(r,H),(d,0)$.  The sum of the tent values over $j=1,\ldots,d-1$ is $dH/2$.  Hence
\[
  \Phi(\delta)=2\sum_{j=1}^{d-1}S_j(n)\ge dH(n).
\]
Equality occurs precisely when the partial-sum graph is this tent function.  Equivalently,
the slopes $n_j=S_j-S_{j-1}$ are constant for $j\le r$ and constant for $j>r$, so the only
nonzero difference $n_j-n_{j+1}$ is $\delta_r$.  This is exactly the ray $\RR_{>0}e_r$.
\end{proof}

For later use we record the same inequality in coordinates on a chamber where a fixed
partial sum realizes the maximum.

\begin{lemma}[Transverse gap]\label{lem:transverse-gap}
Fix $1\le j\le d-1$.  For every $\delta\in\RR_{\ge0}^{d-1}$,
\begin{equation}\label{eq:Rj-formula}
  \Phi(\delta)-dS_j(\delta)
  =
  \sum_{i<j}i(j-i)\delta_i+
  \sum_{i>j}(i-j)(d-i)\delta_i.
\end{equation}
In particular this linear form is nonnegative and vanishes exactly on the ray
$\RR_{\ge0}e_j$.
\end{lemma}

\begin{proof}
Using \eqref{eq:c_i} and \eqref{eq:inverse-Cartan-coeff}, the coefficient of $\delta_i$ in
$\Phi-dS_j$ is $c_i-b_{ji}$.  If $i\le j$, this is
\[
  i(d-i)-i(d-j)=i(j-i),
\]
and if $i\ge j$, it is
\[
  i(d-i)-j(d-i)=(i-j)(d-i).
\]
This proves the formula.
\end{proof}

\section{Integrability and cusp tails}

\subsection{The critical integrability exponent}

The modular inequality now gives both the exact \(L^r\)-threshold and the decay
rate of the alpha cusp.  We begin with integrability.

Since $\nu$ is the quotient measure obtained from $\vol(K)=1$, equations
\eqref{eq:weight-delta} and \eqref{eq:alpha-qH} give, for real $r>0$,
\begin{equation}\label{eq:alpha-integral-sum}
  \int_Y\alpha(x)^r\,d\nu(x)
  =
  \sum_{\delta\in\ZZ_{\ge0}^{d-1}}
  C(\supp\delta)q^{rH(\delta)-\Phi(\delta)}.
\end{equation}
The term $\delta=0$ is finite.

\begin{proof}[Proof of Theorem~\ref{thm:integrability}]
Assume $0<r<d$.  Since there are only finitely many support patterns, the constants
$C(\supp\delta)$ are uniformly bounded above.  By Lemma~\ref{lem:Phi-dH},
\[
  rH(\delta)-\Phi(\delta)
  \le -\left(1-\frac rd\right)\Phi(\delta).
\]
Therefore the right side of \eqref{eq:alpha-integral-sum} is bounded by a constant times
\[
  \sum_{\delta\in\ZZ_{\ge0}^{d-1}}
  q^{-(1-r/d)\sum_{i=1}^{d-1}i(d-i)\delta_i},
\]
which is a product of convergent geometric series.

For divergence at $r\ge d$, fix $1\le a\le d-1$ and take $\delta=Ne_a$.  By
Lemma~\ref{lem:Phi-dH}, or directly from \eqref{eq:inverse-Cartan-coeff},
\[
  H(Ne_a)=\frac{a(d-a)}dN,
  \qquad
  \Phi(Ne_a)=a(d-a)N=dH(Ne_a).
\]
The support is fixed, so $C(\supp(Ne_a))$ is independent of $N$.  The corresponding
subseries is
\[
  \sum_{N\ge1}q^{(r-d)a(d-a)N/d},
\]
which diverges for $r\ge d$.  This proves the criterion for $\nu$, and the statement for
$\mu$ follows because $\mu$ is a constant multiple of $\nu$.
\end{proof}

\subsection{The sharp cusp-tail estimate}

The same comparison also determines the power-law size of the cusp.

For $T\ge1$, define the alpha-cusp set
\begin{equation}\label{eq:ET-def}
  E_T=\{x\in Y:\alpha(x)>T\}.
\end{equation}
Put
\begin{equation}\label{eq:L-def}
  L=\log_qT.
\end{equation}

\begin{theorem}[Sharp tail exponent]\label{thm:sharp-tail}
There exist constants $0<A_1\le A_2<\infty$, depending only on $d$ and $q$, such that
\begin{equation}\label{eq:sharp-tail}
  A_1T^{-d}\le \nu(E_T)\le A_2T^{-d}
  \qquad (T\ge1).
\end{equation}
The same estimate holds for the probability measure $\mu$ after multiplying the constants
by $\nu(Y)^{-1}$.
\end{theorem}

\begin{proof}
Let $C_*=\max_I C(I)$, where $I$ runs over all subsets of $\{1,\ldots,d-1\}$.  The upper
bound follows by decomposing according to a partial sum that realizes $H$.  For each
$\delta$ with $H(\delta)>L$, choose one index $j$ such that $H(\delta)=S_j(\delta)$.  Then
\[
  \nu(E_T)
  \le
  C_*\sum_{j=1}^{d-1}
  \sum_{\substack{\delta\in\ZZ_{\ge0}^{d-1}\\ S_j(\delta)>L}}
  q^{-\Phi(\delta)}.
\]
By Lemma~\ref{lem:transverse-gap}, write
\[
  \Phi(\delta)=dS_j(\delta)+R_j(\delta),
\]
where
\[
  R_j(\delta)=\sum_{i<j}i(j-i)\delta_i+
  \sum_{i>j}(i-j)(d-i)\delta_i.
\]
The form $R_j$ does not involve $\delta_j$ and is strictly positive in every transverse
coordinate.

Fix the transverse vector $z=(\delta_i)_{i\ne j}$.  Let
\[
  M_j(z)=\sum_{i\ne j}b_{ji}z_i,
  \qquad c_j=j(d-j).
\]
Since $dS_j(\delta)=c_j\delta_j+M_j(z)$, the sum over all $\delta_j\ge0$ satisfying
$S_j(\delta)>L$ is bounded by
\[
  \sum_{\delta_j:\,c_j\delta_j+M_j(z)>dL}
  q^{-c_j\delta_j-M_j(z)}q^{-R_j(z)}
  \le C_jq^{-dL}q^{-R_j(z)}
\]
for a constant $C_j$ depending only on $d$.  Summing over all transverse $z$ gives a finite
geometric product because all nonzero coefficients of $R_j$ are positive.  Hence
$\nu(E_T)\ll q^{-dL}=T^{-d}$.

For the lower bound, use the ray $\delta=Ne_1$.  Choose
\[
  N_T=\left\lfloor\frac{dL}{d-1}\right\rfloor+1.
\]
Then $H(N_Te_1)>L$, so this single vertex lies in $E_T$.  Its measure is
\[
  C(\{1\})q^{-(d-1)N_T}\gg q^{-dL}=T^{-d}.
\]
This proves \eqref{eq:sharp-tail}.
\end{proof}

\begin{remark}[Discrete oscillation]
Because $dH(\delta)\in\ZZ$, the function $T^d\nu(E_T)$ need not converge as
$T\to\infty$.  Theorem~\ref{thm:sharp-tail} proves that it remains bounded above and below
by positive constants.  The rational-cone decomposition further suggests a refinement in
terms of periodic functions of $\log_qT$; obtaining a complete periodic expansion, including
all boundary contributions, is left as a separate problem.
\end{remark}

\section{The height zeta function}\label{sec:hzf}

We define the unnormalized height zeta function, using the positive-moment
convention, by
\begin{equation}\label{eq:zeta-s}
  Z_\alpha(s)=\int_Y\alpha(x)^s\,d\nu(x)
  \qquad (s\in\CC).
\end{equation}
Equivalently, \(Z_\alpha\) is the Mellin transform of the cusp-height distribution
determined by \(\alpha\) and \(\nu\).  Our positive-exponent convention differs from the
inverse-height convention commonly used for global height zeta functions that count
rational points.  It should also be distinguished from the Euler-product edge and chamber
zeta functions previously studied for the same standard quotient in the case \(d=3\)
\cite{HongKwonEdgeZeta,HongKwonChamberZeta}.
The probability-normalized height zeta function is $Z_\alpha(s)/\nu(Y)$.  Set
\begin{equation}\label{eq:u-variable}
  u=q^{s/d}=\exp\bigl((s/d)\log q\bigr),
\end{equation}
where $\log q$ is the positive real logarithm.  Since $dH(\delta)$ is an integer by
\eqref{eq:inverse-Cartan-coeff}, the zeta function is
\begin{equation}\label{eq:zeta-delta}
  Z_\alpha(s)
  =
  \sum_{\delta\in\ZZ_{\ge0}^{d-1}}
  C(\supp\delta)q^{-\Phi(\delta)}u^{dH(\delta)}.
\end{equation}
For $1\le j\le d-1$, let
\begin{equation}\label{eq:b-vector}
  b_j=(b_{j1},\ldots,b_{j,d-1}),
  \qquad
  b_{ji}=\min(i,j)(d-\max(i,j)),
\end{equation}
so that $dS_j(\delta)=\langle b_j,\delta\rangle$.

We first derive a rational-cone formula and its critical pole, and then specialize
the result to some low-dimensional cases.

\subsection{Rational-cone formula}

\begin{theorem}[Rational-cone formula]\label{thm:zeta-rational}
The function $Z_\alpha(s)$ is a rational function of $u=q^{s/d}$.  More explicitly, for a
nonempty subset $I\subset\{1,\ldots,d-1\}$ we regard a vector $\delta\in\ZZ^I$ as a vector
in $\ZZ^{d-1}$ by setting $\delta_i=0$ for $i\notin I$.  For $1\le j\le d-1$, define the
rational partition polyhedron
\begin{align}
  P_{I,j}=\bigl\{\delta\in\RR^I:\;&\delta_i\ge1\quad (i\in I),\nonumber\\
  &\langle b_j-b_\ell,\delta\rangle\ge1\quad (1\le\ell<j),\nonumber\\
  &\langle b_j-b_\ell,\delta\rangle\ge0\quad (j<\ell\le d-1)\bigr\}.\label{eq:PIj-new}
\end{align}
Thus $P_{I,j}$ consists of those integer points with support $I$ for which $j$ is the
smallest index attaining the maximum of the partial sums.  The strict inequalities for
$\ell<j$ are written as $\ge1$ because all quantities
$\langle b_j-b_\ell,\delta\rangle=d(S_j(\delta)-S_\ell(\delta))$ are integers.  Then
\begin{equation}\label{eq:zeta-cone-general}
  Z_\alpha(s)
  =C(\varnothing)+
  \sum_{\varnothing\ne I\subset\{1,\ldots,d-1\}}C(I)
  \sum_{j=1}^{d-1}
  \sum_{\delta\in P_{I,j}\cap\ZZ^I}
  \prod_{i\in I}(q^{-c_i}u^{b_{ji}})^{\delta_i}.
\end{equation}
Each inner lattice-point series is rational.
\end{theorem}

\begin{proof}
For a fixed support $I$, the coefficient $C(I)$ is constant.  The integer vector
$\delta$ belongs to exactly one of the sets $P_{I,j}$, namely the one where $j$ is the
smallest maximizer of $S_1(\delta),\ldots,S_{d-1}(\delta)$.  On this region,
\[
  dH(\delta)=dS_j(\delta)=\langle b_j,\delta\rangle.
\]
Thus the summand in \eqref{eq:zeta-delta} is a monomial in $q$ and $u$.  The conditions
\eqref{eq:PIj-new} define a rational partition polyhedron in the positive orthant.  The
lattice-point generating function of a rational partition polyhedron is rational; for
example, one triangulates the recession cone into rational simplicial cones and writes the
corresponding partition parallelepipeds as finite sums of monomials divided by products of
terms of the form $1-M$.  Applying this to every pair $(I,j)$ gives
\eqref{eq:zeta-cone-general}.
This standard rational-cone generating-function argument may also be found in
Beck--Robins~\cite{BeckRobins}.
\end{proof}

\subsection{Convergence and the critical pole coefficient}

\begin{corollary}[Abscissa of convergence and pole]\label{cor:zeta-abscissa}
The series \eqref{eq:zeta-s} converges absolutely if and only if
\[
  \operatorname{Re}(s)<d.
\]
It admits a meromorphic continuation to all $s\in\CC$ as a rational function of $q^{s/d}$.
At $s=d$ it has a simple pole on the real axis.  Equivalently, the coefficient of
$(d-s)^{-1}$ is positive.
\end{corollary}

\begin{proof}
Absolute convergence for $\operatorname{Re}(s)<d$ is exactly the integrability argument in
Theorem~\ref{thm:integrability}, with $r=\operatorname{Re}(s)$.  If $\operatorname{Re}(s)\ge d$, the
rank-one subseries $\delta=Ne_j$ has ratio
\[
  q^{-j(d-j)}q^{s j(d-j)/d},
\]
whose modulus is at least $1$.  Hence the terms in this subseries do not tend to zero in
absolute value, and the original series cannot converge absolutely beyond the half-plane
$\operatorname{Re}(s)<d$.

Meromorphic continuation follows from Theorem~\ref{thm:zeta-rational}.  To see that the pole at
$s=d$ is simple, fix $j$ and write $\delta=Ne_j+z$, where
$z\in\ZZ_{\ge0}^{\{1,\ldots,d-1\}\setminus\{j\}}$ is transverse to the $j$-th ray.  For
all sufficiently large $N$ the maximum is attained at $S_j$.  The selected chamber condition
is eventually stable because the coefficient of $N$ in $S_j(Ne_j+z)$ is strictly larger
than the coefficient of $N$ in $S_\ell(Ne_j+z)$ for every $\ell\ne j$.  Lemma~\ref{lem:transverse-gap}
gives
\[
  \Phi(Ne_j+z)-dS_j(Ne_j+z)=R_j(z),
\]
where $R_j$ is strictly positive in every transverse coordinate.  Therefore the singular
part is a finite sum over $j$ of one-dimensional geometric series in $N$, with coefficients
summed against the convergent transverse weight $q^{-R_j(z)}$.  Thus the pole at the
positive real point $s=d$ has order one, and its coefficient in the expansion in powers of
$(d-s)^{-1}$ is positive.
\end{proof}

\begin{corollary}[Critical pole coefficient]\label{cor:critical-pole-coefficient}
For $1\le j\le d-1$ and
$z=(z_i)_{i\ne j}\in\ZZ_{\ge0}^{\{1,\ldots,d-1\}\setminus\{j\}}$, define
\begin{equation}\label{eq:Rj-z}
  R_j(z)=\sum_{i<j}i(j-i)z_i+\sum_{i>j}(i-j)(d-i)z_i.
\end{equation}
Then, as $s\to d$ through real values with $s<d$,
\begin{equation}\label{eq:critical-pole-expansion}
  Z_\alpha(s)
  =\frac{1}{d-s}\frac{d}{\log q}
  \sum_{j=1}^{d-1}\frac1{j(d-j)}
  \sum_{z}
  C\bigl(\{j\}\cup\supp z\bigr)q^{-R_j(z)}
  +O(1),
\end{equation}
where the inner sum is over all such transverse vectors $z$.  Equivalently,
\begin{equation}\label{eq:critical-pole-limit}
  -\operatorname{Res}_{s=d}Z_\alpha(s)
  =\lim_{s\to d^-}(d-s)Z_\alpha(s)
  =\frac{d}{\log q}
  \sum_{j=1}^{d-1}\frac1{j(d-j)}
  \sum_z C\bigl(\{j\}\cup\supp z\bigr)q^{-R_j(z)}.
\end{equation}
\end{corollary}

\begin{proof}
Fix $j$ and a transverse vector $z$.  Put $c_j=j(d-j)$ and write
$\delta=Ne_j+z$.  For all sufficiently large $N$, the partial sum $S_j$ is the smallest
maximizing chamber selected in Theorem~\ref{thm:zeta-rational}.  This eventual stability follows
from the same strict comparison of the coefficients of $N$ in the linear forms
$S_1,\ldots,S_{d-1}$.  By Lemma~\ref{lem:transverse-gap},
\[
  \Phi(Ne_j+z)-dS_j(Ne_j+z)=R_j(z),
\]
and this expression is independent of $N$.  The tail of the corresponding one-dimensional
series is therefore
\[
  C(\{j\}\cup\supp z)q^{-R_j(z)}
  \sum_{N\ge N_0(z)}q^{(s-d)c_jN/d}.
\]
Changing the finite lower limit $N_0(z)$ affects only the holomorphic part at $s=d$.
Indeed, $N_0(z)$ grows at most linearly in $\sum_{i\ne j}z_i$, because the inequalities
which force $S_j$ to be the selected maximal partial sum are linear in $N$ and $z$.  Since
$R_j(z)$ has strictly positive coefficients in every transverse coordinate, both
\[
  \sum_z q^{-R_j(z)}
  \qquad\text{and}\qquad
  \sum_z N_0(z)q^{-R_j(z)}
\]
converge.  Therefore the finite initial segments may be summed over all transverse $z$ and
contribute only an $O(1)$ term near $s=d$.  The singular part of the tail is consequently
\[
  \frac{d}{(d-s)c_j\log q}
  C(\{j\}\cup\supp z)q^{-R_j(z)}.
\]
Summing this absolutely convergent transverse contribution over $j$ and $z$ gives
\eqref{eq:critical-pole-expansion}.
\end{proof}

\begin{example}[Low-dimensional pole coefficients]\label{ex:low-dimensional-pole-coefficients}
The critical pole-coefficient formula agrees with the explicit rational functions in
Section~\ref{sec:low-rank-zeta}.  For $d=3,4,5$ one obtains respectively
\[
  \lim_{s\to3^-}(3-s)Z_{\alpha,3}(s)
  =\frac{6}{q^2(q-1)^3(q+1)\log q},
\]
\[
  \lim_{s\to4^-}(4-s)Z_{\alpha,4}(s)
  =\frac{4(q^2+3q+1)}
  {q^4(q-1)^5(q+1)^2(q^2+q+1)\log q},
\]
and
\[
  \lim_{s\to5^-}(5-s)Z_{\alpha,5}(s)
  =\frac{10(q^4+q^3+3q^2+q+1)}
  {q^6(q-1)^7(q+1)^3(q^2+1)(q^2+q+1)^2\log q}.
\]
\end{example}

\subsection{Explicit formulas for \(d=3,4,5\)}\label{sec:low-rank-zeta}

We conclude the section by collecting the rational functions in the first three
nontrivial cases.

For clarity we now write $Z_{\alpha,d}$ for the height zeta function in the
case of matrix size \(d\).  In each
of the following formulas we put
\[
  u=q^{s/d}
\]
with the displayed value of $d$.  The rational-cone formula of
Theorem~\ref{thm:zeta-rational} gives a finite sum of geometric series; after collecting over a
common denominator, one convenient form is
\[
  Z_{\alpha,d}(s)=\frac{P_d(q,u)}{D_d(q,u)}.
\]

\begin{proposition}[The cases $d=3,4,5$]\label{prop:explicit-low-rank-zeta}
For $d=3$,
\begin{align*}
  D_3(q,u)&=(q-1)^2(q+1)(q^2+q+1)(q-u)(q^4-u^3),\\
  P_3(q,u)&=q^2-qu+2(q^2+q+1)u^2-qu^3+u^4.
\end{align*}
Thus $Z_{\alpha,3}(s)=P_3(q,u)/D_3(q,u)$, with $u=q^{s/3}$.

For $d=4$,
\begin{align*}
  D_4(q,u)={}&(q-1)^3(q+1)^2(q^2+1)(q^2+q+1)\\
  &\cdot(q-u)(q+u)(q^3-u^2)(q^3+u^2)(q^5-u^4),
\end{align*}
with $u=q^{s/4}$, and
\begin{align*}
  P_4(q,u)={}&q^7-q^5u^2+2q^4(q+1)(q^2+1)u^3\\
  &+q^2(q^5+2q^4+4q^3+4q^2+4q+1)u^4\\
  &-2q^2(q+1)(q^2+1)u^5\\
  &+(q^5+4q^4+4q^3+4q^2+2q+1)u^6\\
  &+2(q+1)(q^2+1)u^7-q^2u^8+u^{10}.
\end{align*}
Then $Z_{\alpha,4}(s)=P_4(q,u)/D_4(q,u)$.

For $d=5$,
\begin{align*}
  D_5(q,u)={}&(q-1)^4(q+1)^2(q^2+1)(q^2+q+1)(q^4+q^3+q^2+q+1)\\
  &\cdot(q-u)(q^6-u^5)(q^7-u^5)(q^8-u^5),
\end{align*}
with $u=q^{s/5}$, and
\begin{align*}
  P_5(q,u)={}&q^{12}-q^{11}u+2q^8(q^4+q^3+q^2+q+1)u^4\\
  &-q^5(q^6-q^4-2q^3-2q^2-2q-1)u^5\\
  &+q^4(2q^8+2q^7+3q^6+2q^5+q^4-2q^3-2q^2-2q-1)u^6\\
  &-2q^5(q^2+1)(q^4+q^3+q^2+q+1)u^7\\
  &+2q^2(q^4+q^3+q^2+q+1)(q^4+q^3+3q^2+q+1)u^8\\
  &-2q(q^2+1)(q^4+q^3+q^2+q+1)u^9\\
  &+(-q^8-2q^7-2q^6-2q^5+q^4+2q^3+3q^2+2q+2)u^{10}\\
  &+q(q^6+2q^5+2q^4+2q^3+q^2-1)u^{11}\\
  &+2(q^4+q^3+q^2+q+1)u^{12}-qu^{15}+u^{16}.
\end{align*}
Then $Z_{\alpha,5}(s)=P_5(q,u)/D_5(q,u)$.
\end{proposition}

\begin{proof}
Appendix~\ref{app:low-rank-computation} gives an exact finite reduction of the
calculation.  In the case \(d=3\), the computation is written out completely; for
\(d=4,5\), the remaining verification is a finite symbolic identity.  The concavity of
the partial sums reduces the condition that $j$ be the first
maximizer to two inequalities on consecutive slopes.  After the $j$-th coordinate is summed
as a geometric series, the remaining transverse sum is separated into finitely many residue
classes.  This proves that the displayed $D_d(q,u)$ is a common denominator and gives the
degree bounds
\[
  \deg_u\bigl(D_3Z_{\alpha,3}\bigr)\le4,\qquad
  \deg_u\bigl(D_4Z_{\alpha,4}\bigr)\le10,\qquad
  \deg_u\bigl(D_5Z_{\alpha,5}\bigr)\le16.
\]
The coefficient identity in Proposition~\ref{prop:finite-coefficient-verification} then
determines the three numerators uniquely and yields exactly the polynomials displayed above.
As a useful check, substituting $u=1$ in each formula gives
\[
  Z_{\alpha,d}(0)=\nu(Y)
  =
  \frac{d}{\left(\prod_{m=1}^{d}(q^m-1)\right)
            \left(\prod_{m=2}^{d-1}(q^m-1)\right)}
\]
for $d=3,4,5$, in agreement with Theorem~\ref{thm:volume}. The pole coefficients at
\(u=q\) also agree with Example~\ref{ex:low-dimensional-pole-coefficients}.
\end{proof}

\begin{remark}[Positive real poles]
The first positive real pole of each displayed function occurs at $u=q$, i.e. at $s=d$.
This is the pole predicted by Corollaries~\ref{cor:zeta-abscissa} and~\ref{cor:critical-pole-coefficient}.  The remaining
positive real poles lie farther out in the $u$-plane and reflect secondary chamber
directions.
\end{remark}

\section{Further directions}

The preceding unified volume, cusp, and zeta calculations suggest several
follow-up problems.

\begin{enumerate}[label=(\arabic*)]
\item Determine a closed general denominator and numerator pattern for $Z_{\alpha,d}$.
The low-dimensional cases show a surprisingly compact factorization after collecting terms, but the
factor pattern is not yet transparent from the cone formula alone.
\item Refine the tail estimate from two-sided bounds to an exact finite sum of periodic
leading terms in $\log_qT$.  The proof of Theorem~\ref{thm:sharp-tail} already reduces the problem
to finitely many rational cones; the remaining task is to write the leading periodic function
explicitly.
\item Compare the pole coefficient in Corollary~\ref{cor:critical-pole-coefficient} with Eisenstein-series residues
or height-zeta normalizations for the same quotient.  This should clarify how the elementary
building height used here matches automorphic normalizations.
\item Extend the method to $\SL_d$, $\GL_d$, and other split groups.  In that setting the
coordinates $\delta_i$ should be replaced by simple-root coordinates, $\Phi$ by the modular
character $2\rho$, and $H$ by a maximum of normalized fundamental-weight heights.
\end{enumerate}

\begin{appendix}

\section{The two finite q-identities in the cut-set summation}\label{app:q-identities}

This appendix gives the algebra behind Lemma~\ref{lem:closed-F}.  We use the standard notation
\[
  (a;q)_n=\prod_{r=0}^{n-1}(1-aq^r),\qquad (a;q)_0=1.
\]
Recall that $Q_n=(-1)^n(q;q)_n$.
For the standard \(q\)-binomial and basic hypergeometric identities used below, see
Andrews--Askey--Roy~\cite{AndrewsAskeyRoy} and Gasper--Rahman
\cite{GasperRahman}.

\begin{lemma}[A finite $q$-binomial convolution]\label{lem:appendix-q-convolution}
For every integer $n\ge0$,
\begin{equation}\label{eq:appendix-q-convolution}
  \sum_{k=0}^n
  (-1)^kq^{\binom k2}
  \frac{(a;q)_{n-k}}{(q;q)_k(q;q)_{n-k}}
  =\frac{(-1)^nq^{\binom n2}a^n}{(q;q)_n}.
\end{equation}
\end{lemma}

\begin{proof}
We use the two standard $q$-binomial generating functions
\[
  \sum_{k\ge0}\frac{(-1)^kq^{\binom k2}}{(q;q)_k}x^k=(x;q)_\infty,
  \qquad
  \sum_{m\ge0}\frac{(a;q)_m}{(q;q)_m}x^m=\frac{(ax;q)_\infty}{(x;q)_\infty}.
\]
To avoid any convergence ambiguity, first regard $q$ as a complex parameter with
$|q|<1$ and $|x|$ sufficiently small.  Multiplication gives $(ax;q)_\infty$.  Comparing the
coefficient of $x^n$ yields \eqref{eq:appendix-q-convolution} in the domain $|q|<1$.  Both
sides of \eqref{eq:appendix-q-convolution} are rational functions of $a$ and $q$ whose
possible poles are contained in the finite set $(q;q)_0\cdots(q;q)_n=0$.  Hence the identity
extends to every specialization for which these denominators are nonzero.  In particular it
applies to the prime power $q>1$ used in the body of the paper.
\end{proof}

\begin{lemma}[The induction convolution]\label{lem:appendix-first-q-identity}
For $1\le n\le d-1$,
\begin{equation}\label{eq:appendix-first-q-identity}
  \sum_{r=0}^{n-1}
  \frac{q^{\binom r2}Q_{d-r-1}}{Q_rQ_{n-r}}
  =
  \frac{q^{\binom n2}Q_{d-n-1}}{Q_n}\bigl(q^{n(d-n)}-1\bigr).
\end{equation}
\end{lemma}

\begin{proof}
Apply Lemma~\ref{lem:appendix-q-convolution} with $a=q^{d-n}$.  Separate the $k=n$ term and
move it to the right-hand side.  Since
\[
  (q^{d-n};q)_{n-k}=(-1)^{n-k}\frac{Q_{d-k-1}}{Q_{d-n-1}},\qquad
  (q;q)_m=(-1)^mQ_m,
\]
the remaining sum over $0\le k\le n-1$ becomes
\[
  \frac1{Q_{d-n-1}}
  \sum_{k=0}^{n-1}\frac{q^{\binom k2}Q_{d-k-1}}{Q_kQ_{n-k}}.
\]
The right-hand side becomes
$q^{\binom n2}(q^{n(d-n)}-1)/Q_n$.  Multiplying by $Q_{d-n-1}$ proves the identity.
\end{proof}

\begin{lemma}[The terminal identity]\label{lem:appendix-terminal-q-identity}
For $d\ge1$,
\begin{equation}\label{eq:appendix-terminal-q-identity}
  \sum_{n=0}^{d-1}\frac{q^{\binom n2}}{Q_n(q^{d-n}-1)}
  =\frac{dq^{\binom d2}}{Q_d}.
\end{equation}
\end{lemma}

\begin{proof}
Differentiate \eqref{eq:appendix-q-convolution} with $n=d$ with respect to $a$, and then
set $a=1$.  For $m\ge1$,
\[
  \left.\frac{d}{da}(a;q)_m
\right|_{a=1}
  =-\prod_{r=1}^{m-1}(1-q^r)=-(-1)^{m-1}Q_{m-1}.
\]
Thus the derivative of the left-hand side of \eqref{eq:appendix-q-convolution} at $a=1$ is
\[
  \sum_{r=0}^{d-1}\frac{q^{\binom r2}Q_{d-r-1}}{Q_rQ_{d-r}}
  =\sum_{r=0}^{d-1}\frac{q^{\binom r2}}{Q_r(q^{d-r}-1)}.
\]
The derivative of the right-hand side is
$dq^{\binom d2}/Q_d$.  This proves \eqref{eq:appendix-terminal-q-identity}.
\end{proof}

\begingroup

\section{Low-dimensional height zeta computation}\label{app:low-rank-computation}
\par\medskip

We reduce the verification of the low-dimensional formulas to finitely many exact
algebraic identities.  The case \(d=3\) can be checked directly from the displayed
decomposition.  For \(d=4,5\), the residue-class decomposition below leads to a finite
symbolic computation, which may be carried out by hand in principle or, more conveniently,
with exact computer algebra over \(\mathbb Q(q)\).  No numerical approximation is involved.
The first step is a useful reduction of the smallest-maximizer condition to two inequalities.
Besides shortening the low-dimensional calculation, it makes the source of the possible
denominator factors transparent.

\subsection{The first-maximizer band}

Put
\[
  n_j=S_j-S_{j-1}\qquad(1\le j\le d).
\]
The sequence \(S_0,S_1,\ldots,S_d\) is concave, and therefore \(j\) is the smallest index
at which its maximum is attained if and only if
\begin{equation}\label{eq:first-max-slope}
  n_j>0,\qquad n_{j+1}\le0.
\end{equation}
Fix \(1\le j\le d-1\), write \(\delta=ke_j+z\), where \(z_j=0\), and put
\begin{equation}\label{eq:ABw}
  A_j(z)=\sum_{i<j}iz_i,\qquad
  B_j(z)=\sum_{i>j}(d-i)z_i,\qquad
  w_j(z)=A_j(z)-B_j(z).
\end{equation}
Since
\[
  dn_j=(d-j)k+B_j(z)-A_j(z),\qquad
  dn_{j+1}=B_j(z)-A_j(z)-jk,
\]
condition \eqref{eq:first-max-slope} is equivalent to
\begin{equation}\label{eq:first-max-band}
  -jk\le w_j(z)<(d-j)k.
\end{equation}
In particular \(k>0\).  Define
\begin{equation}\label{eq:Njw}
  N_{d,j}(w)=
  \begin{cases}
    \left\lfloor w/(d-j)\right\rfloor+1,&w\ge0,\\[2mm]
    \left\lceil -w/j\right\rceil,&w<0.
  \end{cases}
\end{equation}
Then \(N_{d,j}(w)\) is exactly the least positive integer \(k\) satisfying
\eqref{eq:first-max-band}.

\begin{lemma}[Ray-first summation]\label{lem:ray-first-summation}
Set
\[
  U=q^{-1}u,\qquad X_j=U^{j(d-j)}.
\]
For every \(d\ge2\),
\begin{equation}\label{eq:ray-first-zeta}
  Z_{\alpha,d}(s)
  =
  C(\varnothing)+
  \sum_{j=1}^{d-1}\;
  \sum_{\substack{z\in\ZZ_{\ge0}^{d-1}\\z_j=0}}
  C\bigl(\{j\}\cup\supp z\bigr)q^{-R_j(z)}
  U^{\langle b_j,z\rangle}
  \frac{X_j^{N_{d,j}(w_j(z))}}{1-X_j},
\end{equation}
where \(R_j\) is given by \eqref{eq:Rj-z}.
\end{lemma}

\begin{proof}
Every nonzero \(\delta\) has a unique smallest maximizing index \(j\).  By
\eqref{eq:first-max-band}, after the transverse vector \(z\) has been fixed, the corresponding
points are precisely
\[
  \delta=ke_j+z,\qquad k\ge N_{d,j}(w_j(z)).
\]
On these points, \(H=S_j\), while Lemma~\ref{lem:transverse-gap} gives
\[
  \Phi(ke_j+z)-dS_j(ke_j+z)=R_j(z).
\]
Moreover,
\[
  dS_j(ke_j+z)=j(d-j)k+\langle b_j,z\rangle.
\]
The support is \(\{j\}\cup\supp z\), independently of \(k\), and hence the sum over \(k\)
is the geometric series displayed in \eqref{eq:ray-first-zeta}.  The zero vector contributes
\(C(\varnothing)\).
\end{proof}

\subsection{Finite residue-class reduction}

For \(J\subset\{1,\ldots,d-1\}\setminus\{j\}\), define the following product
\begin{align}
 \Theta_{d,j,J}(y;U)
 ={}&
 \prod_{\substack{i\in J\\i<j}}
 \frac{q^{-i(j-i)}U^{i(d-j)}y^i}
 {1-q^{-i(j-i)}U^{i(d-j)}y^i}
 \prod_{\substack{i\in J\\i>j}}
 \frac{q^{-(i-j)(d-i)}U^{j(d-i)}y^{-(d-i)}}
 {1-q^{-(i-j)(d-i)}U^{j(d-i)}y^{-(d-i)}}.
\label{eq:Theta}
\end{align}
For \(J=\varnothing\), the product is \(1\).  Its coefficient of \(y^w\) is exactly
\[
  \sum_{\substack{\supp z=J\\w_j(z)=w}}
  q^{-R_j(z)}U^{\langle b_j,z\rangle}.
\]
For \(\operatorname{Re}(s)<d\), one has \(|U|<1\), and the indicated geometric
expansions converge absolutely on a nonempty annulus containing \(|y|=1\).  We use
\([y^w]\Theta_{d,j,J}\) for the coefficient in this bilateral Laurent expansion;
equivalently, it is defined by the absolutely convergent weighted sum displayed above.
Consequently \eqref{eq:ray-first-zeta} becomes
\begin{equation}\label{eq:theta-zeta}
 Z_{\alpha,d}(s)
 =
 C(\varnothing)+
 \sum_{j=1}^{d-1}\frac1{1-X_j}
 \sum_{J\subset\{1,\ldots,d-1\}\setminus\{j\}}C(\{j\}\cup J)
 \sum_{w\in\ZZ}[y^w]\Theta_{d,j,J}(y;U)\,
 X_j^{N_{d,j}(w)}.
\end{equation}
This identity already gives a finite exact reduction: every factor in \(\Theta\) is a
geometric series, and \(N_{d,j}(w)\) is explicitly linear after \(w\) is separated by sign
and residue class modulo \(j\) or \(d-j\).

\begin{lemma}[Low-dimensional denominator reduction]\label{lem:low-dimensional-denominator-reduction}
For \(d=3,4,5\), expand the products in \eqref{eq:Theta}, separate \(w\ge0\) and \(w<0\),
and in the two sums write respectively
\[
  w=(d-j)m+r,\quad0\le r<d-j,
  \qquad\text{and}\qquad
  -w=jm+r,\quad1\le r\le j.
\]
The resulting sums are ordinary geometric series.  After pairing \(j\) with \(d-j\) by
the Dynkin reversal \(i\mapsto d-i\), their common denominators and numerator degree bounds
are as follows:
\begin{center}
\renewcommand{\arraystretch}{1.2}
\begin{tabular}{c|c|c}
\(d\)&common denominator&degree bound\\
\hline
\(3\)&\(D_3(q,u)\)&\(4\)\\
\(4\)&\(D_4(q,u)\)&\(10\)\\
\(5\)&\(D_5(q,u)\)&\(16\)
\end{tabular}
\end{center}
Here \(D_3,D_4,D_5\) are exactly the polynomials displayed in
Proposition~\ref{prop:explicit-low-rank-zeta}.  In particular,
\[
  D_d(q,u)Z_{\alpha,d}(s)\in\mathbb Q(q)[u]
\]
with the indicated degree bound.
\end{lemma}

\begin{proof}
Formula \eqref{eq:theta-zeta} contains only finitely many pairs \((j,J)\).  For a fixed
pair, substitute the product \eqref{eq:Theta}.  The exponent of \(y\) is the integer
\(w_j(z)\), while \eqref{eq:Njw} depends only on its sign and on the displayed residue
class.  Fixing the residue class converts every remaining variable sum into a product of
ordinary geometric series.  Collecting the resulting rational functions and cancelling
common factors gives, respectively,
\begin{align*}
d=3:\quad&
(q-u)(q^4-u^3),\\
d=4:\quad&
(q-u)(q+u)(q^3-u^2)(q^3+u^2)(q^5-u^4),\\
d=5:\quad&
(q-u)(q^6-u^5)(q^7-u^5)(q^8-u^5).
\end{align*}
The factors of \(D_d\) involving only \(q\) come from the constants \(C(I)\).
The largest \(u\)-degrees left after multiplication by these common denominators are
\(4,10,16\).  This calculation uses, for \(d=3,4,5\), respectively \(2,6,16\) choices
of \((j,J)\) up to the reversal symmetry, and only the residue classes shown in the
statement.  Thus it is a finite identity of geometric series, rather than an appeal to
coefficient agreement to an unspecified order.
\end{proof}

\subsection{Finite numerator verification}

Write
\[
  D_d(q,u)=\sum_{r=0}^{R_d}d_{d,r}(q)u^r
\]
and set
\begin{equation}\label{eq:Adn}
  A_{d,n}(q)
  =
  [u^n]Z_{\alpha,d}(s)
  =
  \sum_{\substack{\delta\in\ZZ_{\ge0}^{d-1}\\dH(\delta)=n}}
  C(\supp\delta)q^{-\Phi(\delta)}.
\end{equation}
This is a finite sum.  Indeed, \(dH(\delta)\ge i(d-i)\delta_i\), and hence every vector
occurring in \eqref{eq:Adn} satisfies
\begin{equation}\label{eq:delta-finite-bound}
  0\le\delta_i\le\frac{n}{i(d-i)}.
\end{equation}
Here the stated inequality follows from
\[
  dH(\delta)\ge dS_i(\delta)
  =\sum_{k=1}^{d-1}b_{ik}\delta_k
  \ge b_{ii}\delta_i=i(d-i)\delta_i,
\]
since all \(b_{ik}\) and \(\delta_k\) are nonnegative.

\begin{proposition}[Finite coefficient verification]\label{prop:finite-coefficient-verification}
Let
\[
  M_3=4,\qquad M_4=10,\qquad M_5=16.
\]
For \(d=3,4,5\) and \(0\le n\le M_d\), define
\begin{equation}\label{eq:pd-coefficient}
  p_{d,n}(q)
  =
  \sum_{r=0}^{\min(n,R_d)}d_{d,r}(q)A_{d,n-r}(q).
\end{equation}
Using the finite bounds \eqref{eq:delta-finite-bound}, these coefficients are exactly the
coefficients of the polynomial \(P_d(q,u)\) displayed in
Proposition~\ref{prop:explicit-low-rank-zeta}.  Consequently
\[
  D_d(q,u)Z_{\alpha,d}(s)=P_d(q,u)
\]
for \(d=3,4,5\).
\end{proposition}

\begin{proof}
The equality in \eqref{eq:pd-coefficient} is the Cauchy product identity for
\(D_dZ_{\alpha,d}\).  Every \(A_{d,n-r}\) is evaluated by the explicitly bounded finite
sum \eqref{eq:Adn}--\eqref{eq:delta-finite-bound}; inserting \(C(\supp\delta)\) from
\eqref{eq:C-I-Phi} and collecting powers of \(q\) gives the coefficients printed in
\(P_d(q,u)\). For \(d=3\) this is the direct calculation written below.  For
\(d=4,5\), the same finite sums may be evaluated by hand or verified by exact symbolic
algebra over \(\mathbb Q(q)\); no floating-point comparison is used.
Lemma~\ref{lem:low-dimensional-denominator-reduction} shows that \(D_dZ_{\alpha,d}\) is
a polynomial of degree at most \(M_d\).  Therefore these \(M_d+1\) coefficient identities
determine the entire polynomial and prove the assertion.
\end{proof}

\subsection{The case \(d=3\)}

For completeness, n the case \(d=3\) the full calculation fits in one line.  Write
\(\delta=(a,b)\).  The smallest-maximizer convention gives \(H=S_1\) when \(a\ge b\)
and \(H=S_2\) when \(b>a\).  Thus
\begin{align}\label{eq:appendix-d3-decomposition}
  Z_{\alpha,3}(s)={}&C(\varnothing)
  +(C(\{1\})+C(\{2\}))\frac{q^{-2}u^2}{1-q^{-2}u^2}
\nonumber\\
  &+C(\{1,2\})
  \frac{q^{-4}u^3(1+q^{-2}u^2)}
  {(1-q^{-4}u^3)(1-q^{-2}u^2)}.
\end{align}
Combining the terms gives \(P_3(q,u)/D_3(q,u)\).  The preceding lemmas show that the
formulas for \(d=4,5\) follow from the same finite residue-class summation and may be
verified by exact symbolic algebra.

\endgroup
\end{appendix}

\end{document}